\documentclass{article}

\usepackage{PRIMEarxiv}
\usepackage[utf8]{inputenc}
\usepackage[T1]{fontenc}
\usepackage{hyperref}
\usepackage{url}
\usepackage{booktabs}
\usepackage{amsfonts}
\usepackage{nicefrac}
\usepackage{microtype}
\usepackage[T1]{fontenc}
\usepackage[utf8]{inputenc}
\usepackage{lmodern}

\usepackage{amsmath,amssymb,amsthm}
\usepackage{array,booktabs,longtable,multirow}
\usepackage{tikz}
\usepackage{xcolor}
\usepackage{hyperref}
\hypersetup{colorlinks=true,linkcolor=blue!50!black,citecolor=blue!50!black,urlcolor=blue!50!black}
\newtheorem{theorem}{Theorem}[section]
\newtheorem{proposition}[theorem]{Proposition}
\newtheorem{lemma}[theorem]{Lemma}
\newtheorem{corollary}[theorem]{Corollary}
\theoremstyle{remark}
\newtheorem{remark}[theorem]{Remark}

\newtheorem{conjecture}[theorem]{Conjecture}

\title{Critical phase transitions in minimum-energy configurations for the exponential kernel family $e^{-|x-y|^q}$ on the unit interval}
\author{Michael T.M. Emmerich\\Faculty of Information Technology, University of Jyväskylä, Finland\\\texttt{m.t.m.emmerich@jyu.fi}}
\date{March 28, 2026}

\begin{document}
\maketitle

\begin{abstract}
We study the optimal placement of $k$ ordered points on the unit interval for the bounded pair potential
\[
K_q(d)=e^{-d^q}, \qquad q>0.
\]
The family interpolates between strongly cusp-like kernels for $0<q<1$, the threshold kernel $e^{-d}$, and the flatter Gaussian-type regime $q>1$. Our emphasis is on the transition from collision-free minimizers to endpoint-collapsed minimizers. We reformulate the problem in gap variables, record convexity, symmetry, and the Karush-Kuhn-Tucker conditions, and give a short proof that collisions are impossible for $0<q<1$. At the threshold $q=1$ we recover the endpoint-clustering law for $e^{-d}$, while for $q>1$ we identify critical exponents $q_k$ beyond which interior points are no longer optimal. For odd $k$ we derive the exact universal value
\[
q_{2m+1} = \frac{\log(1/(-\log((1+e^{-1})/2)))}{\log 2} \approx 1.396363475,
\]
and for even $k=4,6,\dots,20$ we compute the numerical transition values
\[
\begin{aligned}
&q_4\approx 1.062682507,\quad
q_6\approx 1.155601329,\quad
q_8\approx 1.206132611,\quad
q_{10}\approx 1.238523533,\\
&q_{12}\approx 1.261308114,\quad
q_{14}\approx 1.278305167,\quad
q_{16}\approx 1.291510874,\quad
q_{18}\approx 1.302082885,\\
&q_{20}\approx 1.310744185.
\end{aligned}
\]
We also include comparison tables and diagrams for the kernels $e^{-\sqrt d}$, $e^{-d}$, and $e^{-d^2}$, briefly relate the bounded family to the singular Riesz kernel $d^{-s}$, and identify the $q\to 0^+$ limit with the Fekete/Chebyshev--Lobatto configuration on $[0,1]$.
\end{abstract}

\section{Introduction}

For an ordered $k$-tuple $0\le x_1\le \cdots \le x_k\le 1$, consider the discrete energy
\begin{equation}
E_{k,q}(x_1,\dots,x_k)
:=
\sum_{1\le i<j\le k} e^{-|x_j-x_i|^q},
\qquad q>0.
\label{eq:defenergy}
\end{equation}
The purpose of this note is to document and organize a transition phenomenon in the one-dimensional minimization of \eqref{eq:defenergy}. The parameter $q$ governs the local shape of the kernel near the origin:
\[
e^{-d^q} = 1-d^q+O(d^{2q})\qquad (d\downarrow 0).
\]
Hence $q<1$ corresponds to an infinite local slope, $q=1$ is the borderline case, and $q>1$ produces a kernel that is flat at the origin.

For singular kernels such as the Riesz family $d^{-s}$, collisions are forbidden because the energy diverges at zero; this is the classical setting of discrete energy minimization on rectifiable sets \cite{BorodachovHardinSaff2019,HardinSaff2004,HardinSaff2005}. In contrast, the bounded family \eqref{eq:defenergy} permits collisions in principle. The question is therefore not whether collisions are admissible, but when they become energetically optimal.

The main findings of the present report are the following.
\begin{itemize}
\item For $0<q<1$, minimizers are collision-free. The proof is based on the comparison $\varepsilon^q \gg \varepsilon$ as $\varepsilon\downarrow 0$.
\item At the threshold $q=1$, partial endpoint clustering occurs. Numerically one observes
\[
m(k)=\left\lfloor \frac{k+1}{3}\right\rfloor
\]
coincident points at each endpoint for $k\ge 2$.
\item For $q>1$, one obtains a family of critical exponents $q_k$ such that for $q>q_k$ all minimizers for the tested values $k=3,\dots,20$ are supported only on $\{0,1\}$.
\item As $q\to 0^+$ for fixed $k$, the minimizers converge to the Fekete configuration on $[0,1]$, i.e. to the affine Chebyshev--Lobatto points.
\item For odd $k$, the critical exponent is independent of $k$ and can be derived in closed form. For even $k$, the transition is described by a one-parameter symmetric branch and yields the numerical values listed in Tables~\ref{tab:qk} and \ref{tab:qk-even-extended}.
\end{itemize}

The paper is organized as follows. Section~2 introduces the gap formulation together with convexity, symmetry, and the KKT conditions. Section~3 gives the local argument excluding collisions for $0<q<1$. Section~4 discusses the threshold case $q=1$, while Section~5 derives the critical exponents for $q>1$. Section~6 treats the asymptotic regime $q\to 0^+$, and Section~7 compares the behavior with three representative kernels and with the Riesz family. Section~8 provides numerical studies for gradient flow at different values of $k$ and $q$. Section~9 summarizes the picture and includes a first phase diagram for the $(k,q)$-plane. The appendix then records open questions and collects the illustrative point diagrams.

\section{Gap variables, convexity, symmetry, and KKT conditions}

Let
\[
g_r:=x_{r+1}-x_r,\qquad r=1,\dots,k-1.
\]
Then $g_r\ge 0$ and
\[
x_j-x_i = g_i+\cdots +g_{j-1}.
\]
Writing $g=(g_1,\dots,g_{k-1})$, the energy becomes
\begin{equation}
\mathcal{E}_{k,q}(g)
=
\sum_{1\le i<j\le k}
\exp\!\Bigl(-\bigl(g_i+\cdots +g_{j-1}\bigr)^q\Bigr).
\label{eq:gapenergy}
\end{equation}

\begin{lemma}[Full-span property]
For every $q>0$ and every $k\ge 2$, every minimizer of \eqref{eq:defenergy} satisfies $x_1=0$ and $x_k=1$. Equivalently, in gap variables,
\[
\sum_{r=1}^{k-1} g_r = 1.
\]
\end{lemma}

\begin{proof}
The kernel $d\mapsto e^{-d^q}$ is strictly decreasing on $(0,\infty)$. If $x_1>0$ or $x_k<1$, one can move the leftmost point to $0$ or the rightmost point to $1$ without decreasing any pairwise distance, and with at least one strict increase unless the point is already at the endpoint. Hence the energy strictly decreases.
\end{proof}

\begin{lemma}[Convexity for $q\ge 1$]
For every fixed $k$ and $q\ge 1$, the gap-energy \eqref{eq:gapenergy} is convex on the simplex
\[
\Delta_{k-1}:=\Bigl\{g\in [0,\infty)^{k-1}: \sum_{r=1}^{k-1} g_r=1\Bigr\}.
\]
\end{lemma}

\begin{proof}
For $q\ge 1$, the map $t\mapsto t^q$ is convex and increasing, and $u\mapsto e^{-u}$ is convex and decreasing on $[0,\infty)$. Thus $t\mapsto e^{-t^q}$ is convex. Each term in \eqref{eq:gapenergy} is the composition of this convex function with a linear form in the gaps, and summing preserves convexity.
\end{proof}

\begin{remark}
For $0<q<1$ the map $t\mapsto e^{-t^q}$ is no longer convex on the whole half-line, but the local argument excluding collisions below uses only the short-distance asymptotics and not global convexity.
\end{remark}

Reflection about the midpoint sends
\[
(x_1,\dots,x_k)\mapsto (1-x_k,\dots,1-x_1)
\]
and preserves the energy. In gap variables this is the reversal
\[
(g_1,\dots,g_{k-1})\mapsto (g_{k-1},\dots,g_1).
\]
Therefore symmetric minimizers exist whenever convexity permits averaging.

\begin{lemma}[KKT condition]
Assume $q\ge 1$ and let $g$ minimize \eqref{eq:gapenergy} on $\Delta_{k-1}$. Then there exists $\lambda\in \mathbb{R}$ such that for every $r=1,\dots,k-1$,
\[
\sum_{i\le r<j} q\,(x_j-x_i)^{q-1} e^{-(x_j-x_i)^q}
=
\lambda
\quad \text{whenever } g_r>0,
\]
and the left-hand side is at most $\lambda$ whenever $g_r=0$.
\end{lemma}

\begin{proof}
This is the standard Karush-Kuhn-Tucker condition for minimization on the simplex. Differentiating \eqref{eq:gapenergy} with respect to $g_r$ gives
\[
\frac{\partial \mathcal{E}_{k,q}}{\partial g_r}
=
-\sum_{i\le r<j} q\,(x_j-x_i)^{q-1}e^{-(x_j-x_i)^q}.
\]
\end{proof}

\section{The local transition at the origin}

The asymptotic expansion
\[
e^{-d^q} = 1-d^q+O(d^{2q})\qquad (d\downarrow 0)
\]
suggests the fundamental dichotomy between $0<q<1$ and $q\ge 1$.

\begin{proposition}[No collisions for $0<q<1$]
Fix $k\ge 2$ and $0<q<1$. Then no minimizer of \eqref{eq:defenergy} has a zero gap. In particular,
\[
0=x_1 < x_2 < \cdots < x_k =1.
\]
\end{proposition}

\begin{proof}
Assume a minimizer has a zero gap, say $g_r=0$. Because the full span equals $1$, there exists some index $s$ with $g_s>0$. For $\varepsilon>0$ small, define a perturbed gap vector $g^{(\varepsilon)}$ by increasing $g_r$ to $\varepsilon$ and decreasing $g_s$ by $\varepsilon$, leaving all other gaps unchanged. This keeps the total sum equal to $1$.

Pairs that cross the formerly zero gap $g_r$ gain at least
\[
1-e^{-\varepsilon^q}=\varepsilon^q+O(\varepsilon^{2q}).
\]
Hence the total energy decreases by a quantity bounded below by $c_1\varepsilon^q$ for some $c_1>0$.

Pairs affected by the decrease of the positive gap $g_s$ lie at a strictly positive distance at $\varepsilon=0$, so their change is differentiable and bounded in magnitude by $c_2\varepsilon$ for some $c_2>0$.

Therefore
\[
E_{k,q}(g^{(\varepsilon)})-E_{k,q}(g)
\le
-c_1\varepsilon^q + c_2\varepsilon + o(\varepsilon).
\]
Since $0<q<1$, one has $\varepsilon^q\gg \varepsilon$ as $\varepsilon\downarrow 0$, and the right-hand side is negative for all sufficiently small $\varepsilon$. This contradicts minimality.
\end{proof}

\begin{remark}[Threshold interpretation]
At $q=1$ the two contributions are both of order $\varepsilon$, so the local argument no longer forces positivity of every gap. This explains why the kernel $e^{-d}$ is the threshold case at which collisions first become possible. For $q>1$, the gain from opening a zero gap is only of order $\varepsilon^q$, which is too small to dominate the global $O(\varepsilon)$ balance.
\end{remark}

\section{The threshold case $q=1$: endpoint clustering}

For $q=1$ we recover the exponential kernel
\[
K(d)=e^{-d}.
\]
The computed minimizers on $[0,1]$ display a remarkably regular endpoint-clustering law. For $k=2,\dots,20$, the numerics support
\[
m(k)=\left\lfloor \frac{k+1}{3}\right\rfloor
\]
coincident points at each boundary, and therefore
\[
z(k)=m(k)-1=\left\lfloor \frac{k-2}{3}\right\rfloor
\]
zero gaps at each side. Equivalently, only a central block of approximately one third of the gaps remains active.

For $k=1,\dots,10$ the minimizing configurations are
\[
\begin{aligned}
&k=1:\ (0),\qquad
k=2:\ (0,1),\qquad
k=3:\ (0,\tfrac12,1),\\
&k=4:\ (0,0.121997,0.878003,1),\\
&k=5:\ (0,0,\tfrac12,1,1),\\
&k=6:\ (0,0,0.251705,0.748295,1,1),\\
&k=7:\ (0,0,0.070484,\tfrac12,0.929516,1,1),\\
&k=8:\ (0,0,0,0.312315,0.687685,1,1,1),\\
&k=9:\ (0,0,0,0.167687,\tfrac12,0.832313,1,1,1),\\
&k=10:\ (0,0,0,0.049548,0.349849,0.650151,0.950452,1,1,1).
\end{aligned}
\]

\section{Critical exponents for $q>1$}

We now ask for which values of $q>1$ interior points cease to be optimal. For each $k\ge 3$, define
\[
q_k:=\inf\Bigl\{q>1:\ \text{all tested minimizers for } E_{k,q}\text{ are supported only on }\{0,1\}\Bigr\}.
\]
The phrase ``supported only on $\{0,1\}$'' means that every minimizing configuration found numerically consists entirely of a balanced split between the two endpoints. We next derive exact and numerical formulas for $q_k$ when $k=3,\dots,20$.

\subsection{Odd $k$: an exact universal critical value}

Let $k=2m+1$ with $m\ge 1$. Consider the symmetric competitor with one interior point,
\[
X_{\mathrm{odd,int}} = ( \underbrace{0,\dots,0}_m,\tfrac12,\underbrace{1,\dots,1}_m),
\]
and the endpoint-only configuration
\[
X_{\mathrm{odd,end}} = (\underbrace{0,\dots,0}_m,\underbrace{1,\dots,1}_{m+1}),
\]
up to reflection.

Their energies are
\[
E_{\mathrm{odd,int}}(q)=m(m-1)+m^2e^{-1}+2m e^{-2^{-q}},
\]
and
\[
E_{\mathrm{odd,end}}(q)=m^2+m(m+1)e^{-1}.
\]
Equating them gives
\[
2e^{-2^{-q}} = 1+e^{-1},
\]
hence
\begin{equation}
q_{2m+1}
=
\frac{\log\!\Bigl(1/\bigl[-\log((1+e^{-1})/2)\bigr]\Bigr)}{\log 2}
\approx 1.396363475.
\label{eq:oddq}
\end{equation}

\begin{proposition}
For every odd $k=2m+1$, the explicit one-midpoint branch
\[
X_{\mathrm{odd,int}}=(0^{(m)},\tfrac12,1^{(m)})
\]
and the balanced endpoint-only branch
\[
X_{\mathrm{odd,end}}=(0^{(m)},1^{(m+1)})
\]
cross exactly at the universal value \eqref{eq:oddq}. More precisely,
\[
E_{\mathrm{odd,int}}(q)-E_{\mathrm{odd,end}}(q)
=
 m\Bigl(2e^{-2^{-q}}-(1+e^{-1})\Bigr),
\]
so the sign of the branch difference is independent of $m$.
\end{proposition}

\begin{proof}
Subtracting the above formulas for $E_{\mathrm{odd,int}}(q)$ and $E_{\mathrm{odd,end}}(q)$ gives
\[
E_{\mathrm{odd,int}}(q)-E_{\mathrm{odd,end}}(q)
=
 m\Bigl(2e^{-2^{-q}}-(1+e^{-1})\Bigr).
\]
Hence the two branches cross exactly when
\[
2e^{-2^{-q}}=1+e^{-1},
\]
which is equivalent to \eqref{eq:oddq}. Since the factor $m$ is positive, the sign of the branch difference is independent of $m$.
\end{proof}

\begin{remark}[Strengthened odd-$k$ evidence]
The proposition proves that the crossing between the two natural odd branches is universal for \emph{all} odd $k$. What remains open is whether this branch crossing is the true global transition for every odd $k$, i.e. whether no competing branch with two or more interior points can have lower energy near the transition.

For the odd values checked numerically beyond Table~\ref{tab:qk}, namely $k=11,13,15,19$, the same transition persists. Writing $q_{\mathrm{odd}}$ for the value in \eqref{eq:oddq}, one finds for example at $k=11$ that
\[
E_{\mathrm{odd,int}}(1.35)\approx 35.95205407 < 36.03638324 \approx E_{\mathrm{odd,end}},
\]
while at $q=1.43$,
\[
E_{\mathrm{odd,int}}(1.43)\approx 36.09652229 > 36.03638324 \approx E_{\mathrm{odd,end}}.
\]
The same sign change is observed for $k=13,15,19$, supporting the conjecture that \eqref{eq:oddq} is the true critical exponent for every odd $k$.
\end{remark}

\subsection{Even $k$: a one-parameter symmetric branch}

Let $k=2m$ with $m\ge 2$. The natural interior branch consists of two symmetric interior points,
\[
X_{\mathrm{even,int}}(a)=
(\underbrace{0,\dots,0}_{m-1},a,1-a,\underbrace{1,\dots,1}_{m-1}),
\qquad 0\le a\le \frac12,
\]
with energy
\begin{align}
E_{\mathrm{even,int}}(a;q)
&=
(m-1)(m-2)+(m-1)^2e^{-1}
\\
&\quad
+2(m-1)\bigl(e^{-a^q}+e^{-(1-a)^q}\bigr)
+e^{-(1-2a)^q}.
\label{eq:evenbranch}
\end{align}
The endpoint-only competitor is
\[
X_{\mathrm{even,end}}=
(\underbrace{0,\dots,0}_{m},\underbrace{1,\dots,1}_{m}),
\]
with energy
\[
E_{\mathrm{even,end}}(q)=m(m-1)+m^2e^{-1}.
\]
Thus the even critical exponent is determined numerically by the equation
\[
\min_{0\le a\le 1/2} E_{\mathrm{even,int}}(a;q)=E_{\mathrm{even,end}}(q).
\]

\begin{table}[t]
\centering
\caption{Critical exponents $q_k$ for $k=3,\dots,10$. For odd $k$ the value is exact and independent of $k$; for even $k$ the value is numerical, obtained by comparing the symmetric interior branch \eqref{eq:evenbranch} with the balanced endpoint-only split.}
\label{tab:qk}
\begin{tabular}{>{$}c<{$} >{$}c<{$} l}
\toprule
k & q_k & transition branch \\
\midrule
3 & 1.396363475 & odd: $(0^{(m)},\frac12,1^{(m)})$ \\
4 & 1.062682507 & even: $(0^{(m-1)},a,1-a,1^{(m-1)})$ \\
5 & 1.396363475 & odd: $(0^{(m)},\frac12,1^{(m)})$ \\
6 & 1.155601329 & even: $(0^{(m-1)},a,1-a,1^{(m-1)})$ \\
7 & 1.396363475 & odd: $(0^{(m)},\frac12,1^{(m)})$ \\
8 & 1.206132611 & even: $(0^{(m-1)},a,1-a,1^{(m-1)})$ \\
9 & 1.396363475 & odd: $(0^{(m)},\frac12,1^{(m)})$ \\
10 & 1.238523533 & even: $(0^{(m-1)},a,1-a,1^{(m-1)})$ \\
\bottomrule
\end{tabular}
\end{table}

\begin{table}[t]
\centering
\caption{Extended numerical values of the even critical exponents $q_{2m}$ up to $k=20$. The data are consistent with monotone increase in the tested range.}
\label{tab:qk-even-extended}
\small
\setlength{\tabcolsep}{7pt}
\renewcommand{\arraystretch}{1.15}
\begin{tabular}{c c c c c c}
\toprule
$k$ & $q_k$ & $k$ & $q_k$ & $k$ & $q_k$ \\
\midrule
4  & 1.062682507 & 10 & 1.238523533 & 16 & 1.291510874 \\
6  & 1.155601329 & 12 & 1.261308114 & 18 & 1.302082885 \\
8  & 1.206132611 & 14 & 1.278305167 & 20 & 1.310744185 \\
\bottomrule
\end{tabular}
\normalsize
\end{table}

The resulting values are displayed in Tables~\ref{tab:qk} and \ref{tab:qk-even-extended}. In particular,
\[
\begin{aligned}
&q_4\approx 1.062682507,\quad
q_6\approx 1.155601329,\quad
q_8\approx 1.206132611,\quad
q_{10}\approx 1.238523533,\\
&q_{12}\approx 1.261308114,\quad
q_{14}\approx 1.278305167,\quad
q_{16}\approx 1.291510874,\quad
q_{18}\approx 1.302082885,\quad
q_{20}\approx 1.310744185.
\end{aligned}
\]
The odd values remain pinned at \eqref{eq:oddq}, whereas the even values increase with $k$ throughout the tested range $4\le k\le 20$.

\section{The asymptotic regime $q\to 0^+$}

The small-$q$ limit does not lead to equally spaced points. Instead, the first nontrivial term in the expansion of the energy selects the classical Fekete configuration on the interval.

\begin{proposition}[Small-$q$ asymptotics]
Fix $k\ge 2$ and let $X=(x_1,\dots,x_k)$ with $0\le x_1<\cdots <x_k\le 1$. Then, as $q\to 0^+$,
\begin{equation}
E_{k,q}(X)
=
\binom{k}{2}e^{-1}
-
q e^{-1}\sum_{1\le i<j\le k}\log |x_j-x_i|
+O(q^2),
\label{eq:smallqexpansion}
\end{equation}
where the $O(q^2)$ term is uniform on compact subsets of the open simplex $0<x_1<\cdots <x_k<1$.
\end{proposition}

\begin{proof}
For every pair with $d_{ij}:=|x_j-x_i|\in (0,1]$,
\[
d_{ij}^q=e^{q\log d_{ij}}=1+q\log d_{ij}+O(q^2)
\qquad (q\to 0^+).
\]
Hence
\[
e^{-d_{ij}^q}
=
 e^{-1}\exp\!\bigl(-q\log d_{ij}+O(q^2)\bigr)
=
 e^{-1}\bigl(1-q\log d_{ij}+O(q^2)\bigr).
\]
Summing over all $1\le i<j\le k$ gives \eqref{eq:smallqexpansion}.
\end{proof}

\begin{corollary}[Fekete-point limit as $q\to 0^+$]
For fixed $k\ge 2$, any family of minimizers of $E_{k,q}$ converges, as $q\to 0^+$, to a maximizer of
\[
\sum_{1\le i<j\le k}\log|x_j-x_i|,
\]
that is, to a Fekete configuration on $[0,1]$. On the interval this configuration is given explicitly by the affine Chebyshev--Lobatto points
\begin{equation}
\xi_j=\frac{1-\cos\!\bigl(j\pi/(k-1)\bigr)}{2},
\qquad j=0,\dots,k-1.
\label{eq:lobatto}
\end{equation}
\end{corollary}

\begin{proof}
The leading constant in \eqref{eq:smallqexpansion} is independent of the configuration. Therefore minimizing $E_{k,q}$ for small $q>0$ is equivalent, to first order, to maximizing $\sum_{i<j}\log|x_j-x_i|$, or equivalently the Vandermonde product $\prod_{i<j}|x_j-x_i|$. On an interval the maximizers are the Fekete points, which are exactly the affine Chebyshev--Lobatto points \eqref{eq:lobatto}.
\end{proof}

Thus the limit $q\to 0^+$ is \emph{not} uniform. For instance, for $k=4$ one obtains the limit configuration $(0,\tfrac14,\tfrac34,1)$, while for $k=5$ one obtains
\[
\Bigl(0,\frac{2-\sqrt2}{4},\frac12,\frac{2+\sqrt2}{4},1\Bigr),
\]
which is visibly more concentrated near the boundary than the uniform grid. As $k\to\infty$, the empirical measures of the points \eqref{eq:lobatto} converge to the arcsine law $\pi^{-1}(x(1-x))^{-1/2}\,dx$, rather than the uniform measure on $[0,1]$.

\begin{figure}[t]
\centering
\begin{tikzpicture}[x=7cm,y=1cm]
  % guide lines
  \draw[gray!60] (0,0) -- (1,0);
  \draw[gray!60] (0,-1.1) -- (1,-1.1);

  % boundary ticks
  \foreach \yy in {0,-1.1}{
    \draw[gray!70] (0,\yy+0.07) -- (0,\yy-0.07);
    \draw[gray!70] (1,\yy+0.07) -- (1,\yy-0.07);
  }

  % uniform points k=8
  \foreach \x in {0,0.142857,0.285714,0.428571,0.571429,0.714286,0.857143,1}{
    \fill[blue!60!black] (\x,0) circle (2pt);
  }

  % lobatto points k=8
  \foreach \x in {0,0.049516,0.188255,0.388740,0.611260,0.811745,0.950484,1}{
    \fill[red!70!black] (\x,-1.1) circle (2pt);
  }

  \node[anchor=east,font=\footnotesize] at (-0.01,0)
    {$\left(\frac{j}{7}\right)_{j=0}^7$};
  \node[anchor=east,font=\footnotesize] at (-0.01,-1.1)
    {$\left(\frac{1-\cos(j\pi/7)}{2}\right)_{j=0}^7$};
  \node[anchor=west,font=\footnotesize] at (1.01,0) {uniform};
  \node[anchor=west,font=\footnotesize] at (1.01,-1.1) {$q\to0^+$};
\end{tikzpicture}
\caption{Comparison for $k=8$: the uniform grid (top) versus the Chebyshev--Lobatto/Fekete limit predicted by $q\to 0^+$ (bottom). The small-$q$ limit is more concentrated near the endpoints and is therefore not uniform.}
\label{fig:smallqcomparison}
\end{figure}
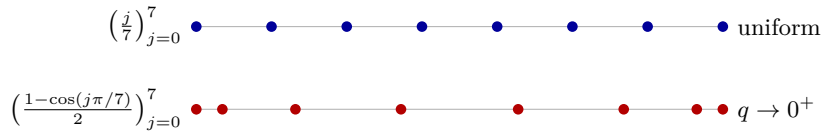

\section{Comparison with selected kernels}

The family $e^{-d^q}$ interpolates between several regimes already encountered in the broader discussion.

\subsection{The cusp-like kernel $e^{-\sqrt d}$ ($q=\frac12$)}

By Proposition~3.1, collisions are impossible for $q=\frac12$. The minimizing configurations for $k=1,\dots,10$ are listed in Table~\ref{tab:half}. They are symmetric and fully distinct, with small but strictly positive edge gaps.

\begin{table}[t]
\centering
\caption{Numerical minimizers for $q=\frac12$, i.e.\ for the kernel $e^{-\sqrt d}$, on the unit interval.}
\label{tab:half}
\begin{tabular}{c l}
\toprule
$k$ & minimizer \\
\midrule
$1$ & $(0)$ \\
$2$ & $(0, 1)$ \\
$3$ & $(0, 0.5, 1)$ \\
$4$ & $(0, 0.254547, 0.745453, 1)$ \\
$5$ & $(0, 0.139429, 0.5, 0.860571, 1)$ \\
$6$ & $(0, 0.081208, 0.337929, 0.662071, 0.918792, 1)$ \\
$7$ & $(0, 0.049814, 0.23337, 0.5, 0.76663, 0.950187, 1)$ \\
$8$ & $(0, 0.03194, 0.164831, 0.379072, 0.620928, 0.835169, 0.96806, 1)$ \\
$9$ & $(0, 0.021274, 0.1189, 0.290198, 0.5, 0.709802, 0.8811, 0.978726, 1)$ \\
$10$ & $(0, 0.014644, 0.087436, 0.224634, 0.403577, 0.596423, 0.775366, 0.912564, 0.985356, 1)$ \\
\bottomrule
\end{tabular}
\end{table}

\subsection{The threshold kernel $e^{-d}$ ($q=1$)}

At $q=1$ collisions first become possible. The first ten minimizing configurations are listed in Table~\ref{tab:one}.

\begin{table}[t]
\centering
\caption{Numerical minimizers for $q=1$, i.e.\ for the kernel $e^{-d}$, on the unit interval.}
\label{tab:one}
\begin{tabular}{c l}
\toprule
$k$ & minimizer \\
\midrule
$1$ & $(0)$ \\
$2$ & $(0, 1)$ \\
$3$ & $(0, 0.5, 1)$ \\
$4$ & $(0, 0.121997, 0.878003, 1)$ \\
$5$ & $(0, 0, 0.5, 1, 1)$ \\
$6$ & $(0, 0, 0.251705, 0.748295, 1, 1)$ \\
$7$ & $(0, 0, 0.070484, 0.5, 0.929516, 1, 1)$ \\
$8$ & $(0, 0, 0, 0.312315, 0.687685, 1, 1, 1)$ \\
$9$ & $(0, 0, 0, 0.167687, 0.5, 0.832313, 1, 1, 1)$ \\
$10$ & $(0, 0, 0, 0.049548, 0.349849, 0.650151, 0.950452, 1, 1, 1)$ \\
\bottomrule
\end{tabular}
\end{table}

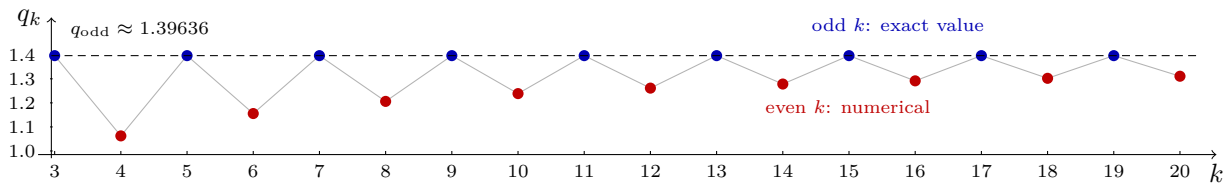
\begin{figure}[t!]
\centering
\begin{tikzpicture}[x=0.88cm,y=3.2cm]
\draw[->] (2.75,0.985) -- (20.55,0.985) node[below] {$k$};
\draw[->] (2.95,0.97) -- (2.95,1.555) node[left] {$q_k$};
\foreach \k in {3,4,...,20}{\draw (\k,0.985) -- (\k,1.000) node[below=2pt,font=\scriptsize] {\k};}
\foreach \y in {1.0,1.1,1.2,1.3,1.4}{\draw (2.93,\y) -- (2.97,\y) node[left=3pt,font=\scriptsize] {\y};}
\draw[gray!60] (3,1.396363475) -- (4,1.062682507) -- (5,1.396363475) -- (6,1.155601329) -- (7,1.396363475) -- (8,1.206132611) -- (9,1.396363475) -- (10,1.238523533) -- (11,1.396363475) -- (12,1.261308114) -- (13,1.396363475) -- (14,1.278305167) -- (15,1.396363475) -- (16,1.291510874) -- (17,1.396363475) -- (18,1.302082885) -- (19,1.396363475) -- (20,1.310744185);
\foreach \k/\q in {3/1.396363475,5/1.396363475,7/1.396363475,9/1.396363475,11/1.396363475,13/1.396363475,15/1.396363475,17/1.396363475,19/1.396363475}{\fill[blue!70!black] (\k,\q) circle (2.1pt);} 
\foreach \k/\q in {4/1.062682507,6/1.155601329,8/1.206132611,10/1.238523533,12/1.261308114,14/1.278305167,16/1.291510874,18/1.302082885,20/1.310744185}{\fill[red!75!black] (\k,\q) circle (2.1pt);} 
\draw[densely dashed] (2.95,1.396363475) -- (20.25,1.396363475);
\node[blue!70!black,anchor=west,font=\scriptsize] at (14.3,1.525) {odd $k$: exact value};
\node[red!75!black,anchor=west,font=\scriptsize] at (13.6,1.185) {even $k$: numerical};
\node[anchor=west,font=\scriptsize] at (3.1,1.505) {$q_{\mathrm{odd}}\approx 1.39636$};
\end{tikzpicture}
\caption{\label{fig:seq}Critical exponents $q_k$ for $k=3,\dots,20$. Odd values are exact and constant; even values are numerical.}
\end{figure}

\subsection{The Gaussian kernel $e^{-d^2}$ ($q=2$)}

For $q=2$, collisions are strongly favored. The numerically minimizing configurations for $k=1,\dots,10$ coincide with the balanced endpoint splits:
\[
(0),\ (0,1),\ (0,1,1),\ (0,0,1,1),\ (0,0,0,1,1),\ \dots
\]
This agrees with the general picture that flatter kernels near the origin make collisions more attractive.

\subsection{Comparison with Riesz energy}

For the singular Riesz kernel $d^{-s}$, $s>0$, collisions are impossible because the pair energy diverges at $d=0$. In that sense the bounded family $e^{-d^q}$ exhibits a genuine bounded-kernel phase transition absent from the Riesz case; see \cite{BorodachovHardinSaff2019,HardinSaff2004,HardinSaff2005} for the classical background on Riesz energies.

\section{Gradient-flow studies for different $(q,k)$ regimes}

To complement the static minimizers, it is helpful to visualize how point configurations move under the energy gradient. For this purpose, we generated a family of gradient-flow pictures in which the horizontal axis represents the unit interval $[0,1]$, while the vertical axis represents an artificial time variable $t$, increasing from top to bottom. Each horizontal slice therefore shows one intermediate configuration of the $k$ points, and the full picture records a discrete gradient-descent path. The resulting plots for $q=0.1$, $q=1$, $q=q_{\mathrm{critical}}(k)$, and $q=2$, with $k=9,10$, are shown in Figure~\ref{fig:gradflowqk}.

For fixed $k$ and $q$, we consider the energy
\[
E_{k,q}(x_1,\dots,x_k)
=
\sum_{1\le i<j\le k} e^{-|x_j-x_i|^q}.
\]
Its gradient can be written in closed form. Since the points are ordered,
\[
x_1\le x_2\le \cdots \le x_k,
\]
we have $|x_j-x_i|=x_j-x_i$ for $i<j$, and differentiation gives
\[
\frac{\partial E_{k,q}}{\partial x_r}
=
q\sum_{j=r+1}^{k} (x_j-x_r)^{q-1}e^{-(x_j-x_r)^q}
-
q\sum_{i=1}^{r-1} (x_r-x_i)^{q-1}e^{-(x_r-x_i)^q},
\qquad r=1,\dots,k.
\]
Equivalently,
\[
\nabla E_{k,q}(x)
=
\left(
\frac{\partial E_{k,q}}{\partial x_1},
\dots,
\frac{\partial E_{k,q}}{\partial x_k}
\right).
\]
We then evolve the points by repeated small steps in the negative gradient direction. If
\[
x^{(n)}=(x_1^{(n)},\dots,x_k^{(n)})
\]
denotes the configuration at step $n$, then the next configuration is obtained from
\[
x^{(n+1)}
\approx
x^{(n)}-\tau \nabla E_{k,q}(x^{(n)}),
\]
followed by a projection back to the ordered unit interval. In the numerical experiments used for the figures, the initial condition is a uniformly spaced configuration shifted slightly into the interior of the interval, so that the subsequent motion is not biased by starting directly at the boundary.

These pictures are not intended as rigorous proofs of global convergence. Rather, they serve as a geometric microscope for the phase transition. They make visible which collisions form first, how rapidly the endpoints attract points, and whether interior points survive for a long time or disappear quickly. In particular, they help distinguish the three regimes already suggested by the analytical results and by the computed minimizers.

For small values of $q$, such as $q=0.1$, the trajectories remain comparatively spread out. The strong cusp of the kernel near the origin produces a strong local resistance against collisions, and the flow therefore preserves a visibly collision-free structure for a long time. This is consistent with the proof that for $0<q<1$ minimizers cannot contain zero gaps.

At the threshold value $q=1$, the pictures begin to show endpoint clustering. Points drift toward the boundary and eventually form stacks at $0$ and $1$, while a smaller active block may remain in the interior. This dynamical picture agrees with the observed endpoint-clustering law and makes the threshold nature of the exponential kernel $e^{-d}$ visually transparent.

For flatter kernels, for example $q=2$, the descent is much more decisive. Points are rapidly absorbed into endpoint clusters, and the surviving interior structure collapses quickly. In this regime the pictures vividly illustrate why the endpoint-supported configurations become energetically favorable.

The most interesting panels are those near the critical exponents. For odd $k$, and in particular for $k=9$, the picture at
\[
q_{\mathrm{odd}}
=
\frac{\log\!\bigl(1/[-\log((1+e^{-1})/2)]\bigr)}{\log 2}
\approx 1.396363475
\]
shows the system poised between two competing tendencies: one toward a one-midpoint configuration and the other toward complete endpoint concentration. It is important, however, to interpret the gradient correctly. In the point variables, the local gradient at a non-endpoint point still depends on the current configuration through the distances to the other points, so it does not literally become constant. What does become uniform at equilibrium is the gradient in the gap variables on the active gaps, as expressed by the Karush--Kuhn--Tucker condition. Equivalently, neighboring active gap derivatives become equal, and therefore the corresponding point-gradient vanishes. This explains why, near the odd critical value, the trajectories in the active interior block appear almost straight: the force balance there is nearly uniform, even though the pointwise gradient still depends on the local geometry. Likewise, for even $k$, the corresponding critical pictures illustrate the competition between a symmetric interior branch and the endpoint-only branch. In this sense, the gradient-flow plots provide a dynamical visualization of the same transition that is detected analytically through the branch comparisons.

\section{Summary and conclusions}

\begin{figure}[h!]
\centering
\begin{tikzpicture}[x=4.4cm,y=0.34cm]
% axes
\fill[green!8] (0.00,3) rectangle (1.00,20.4);
\fill[orange!10] (1.00,3) rectangle (2.05,20.4);
\draw[->] (0,3) -- (2.12,3) node[below] {$q$};
\draw[->] (0,3) -- (0,20.8) node[left] {$k$};
\foreach \q/\lbl in {0/0,0.5/0.5,1/1,1.5/1.5,2/2}{\draw (\q,3) -- (\q,2.8) node[below=2pt,font=\scriptsize] {\lbl};}
\foreach \k in {4,6,8,10,12,14,16,18,20}{\draw (0.00,\k) -- (-0.03,\k) node[left=3pt,font=\scriptsize] {\k};}
\foreach \k in {3,5,7,9,11,13,15,17,19}{\draw (0.00,\k) -- (-0.02,\k);} 

% threshold markers
\draw[densely dashed,thick] (1.00,3) -- (1.00,20.2);
\node[anchor=west,font=\scriptsize] at (1.02,19.7) {$q=1$};
\draw[densely dashed,blue!70!black] (1.396363475,3) -- (1.396363475,20.2);
\node[anchor=west,font=\scriptsize,blue!70!black] at (1.41,18.6) {$q_{\mathrm{odd}}\approx 1.39636$};

% odd points
\foreach \k in {3,5,7,9,11,13,15,17,19}{\fill[blue!70!black] (1.396363475,\k) circle (1.6pt);} 
% even points
\draw[red!75!black,thick] (1.062682507,4) -- (1.155601329,6) -- (1.206132611,8) -- (1.238523533,10) -- (1.261308114,12) -- (1.278305167,14) -- (1.291510874,16) -- (1.302082885,18) -- (1.310744185,20);
\foreach \q/\k in {1.062682507/4,1.155601329/6,1.206132611/8,1.238523533/10,1.261308114/12,1.278305167/14,1.291510874/16,1.302082885/18,1.310744185/20}{\fill[red!75!black] (\q,\k) circle (1.8pt);} 

% annotations and tentative regions
\node[align=left,anchor=west,font=\scriptsize] at (0.10,19.5) {collision-free\\ regime $(0<q<1)$};
\node[align=left,anchor=west,font=\tiny] at (0.99,15.2) {partial\\ clustering};
\node[align=left,anchor=west,font=\scriptsize] at (1.47,8.0) {endpoint-only\\ states observed};
\draw[->,thin] (1.50,7.95) -- (1.36,10.5);

% simple legend
\fill[blue!70!black] (0.14,4.25) circle (1.6pt);
\node[anchor=west,font=\scriptsize] at (0.18,4.25) {odd critical values};
\fill[red!75!black] (0.14,3.55) circle (1.6pt);
\node[anchor=west,font=\scriptsize] at (0.18,3.55) {even critical values};
\end{tikzpicture}
\caption{A first attempt at a $(q,k)$ phase diagram for the kernel $e^{-|x-y|^q}$ on $[0,1]$. The green region marks the proved collision-free regime $0<q<1$. The blue vertical line at $q_{\mathrm{odd}}$ records the exact branch crossing for odd $k$, while the red points show the computed even critical values for $k=4,6,\dots,20$.}
\label{fig:phase-diagram}
\end{figure}
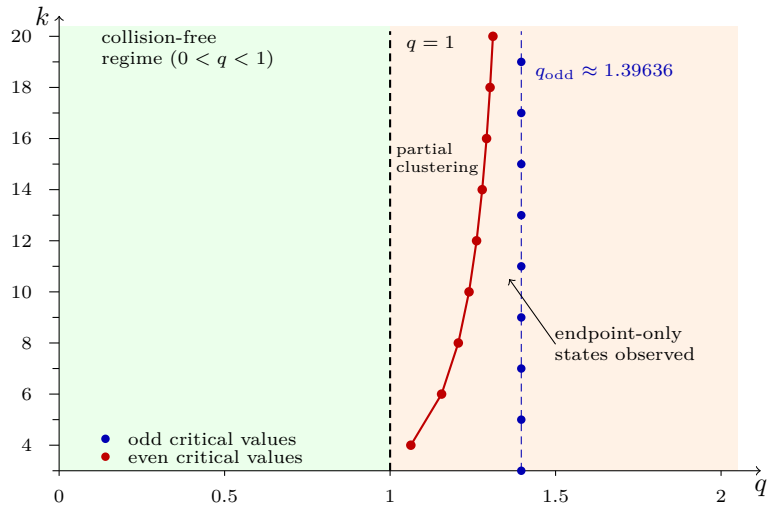

The one-dimensional family $e^{-|x-y|^q}$ exhibits a sharp and surprisingly rich transition in the collision structure of minimizers.
\begin{itemize}
\item For $0<q<1$, the singular local slope prevents all collisions.
\item At $q=1$, endpoint clustering first appears and follows a simple law in the tested range.
\item For $q>1$, interior points persist only up to critical exponents $q_k$.
\item For odd $k$, the critical exponent is the exact universal value \eqref{eq:oddq}.
\item For even $k$, the critical exponent is described by the symmetric two-interior-point branch and increases with $k$ for the tested values $k\le 20$.
\item As $q\to 0^+$, the minimizing configurations converge to the Chebyshev--Lobatto/Fekete points rather than to the uniform grid.
\end{itemize}

A first phase diagram in the $(k,q)$-plane is shown in Figure~\ref{fig:phase-diagram}. It summarizes the proved collision-free regime $0<q<1$, the threshold line $q=1$, the universal odd critical value \eqref{eq:oddq}, and the numerically computed even critical values up to $k=20$.

Several natural questions remain open. Is the odd critical value \eqref{eq:oddq} valid for all odd $k$? Does the sequence $(q_{2m})_m$ increase monotonically? Can one characterize the full large-$k$ phase diagram in the $(k,q)$-plane? These appear to be tractable questions for a fuller paper.

For further study and result verification we refer to \cite{EmmerichBlog2026} Includes an open-source online Python simulator for exploring numerically computed minimizing configurations for different values of $k$ and $q$.

An accompanying GitHub repository with Python code for the numerical results is available at
\begin{center}
\url{https://github.com/emmerichmtm/phaseTransitionExpKernelsUnitLine}
\end{center}

\appendix

\section{Appendix: Open questions}

This appendix records three natural follow-up questions suggested by the explicit formulas and numerics above. Each question appears approachable because the one-dimensional problem admits a gap formulation, a strong symmetry reduction, and for $q\ge 1$ a convex structure on the simplex of admissible gaps.

\subsection{Does the odd critical value persist for all odd $k$?}

Proposition~4.1 proves that the branch crossing between the one-midpoint odd configuration and the endpoint-only odd configuration occurs at the universal value \eqref{eq:oddq} for every odd $k=2m+1$. The remaining issue is global optimality: can one exclude the existence of a lower-energy competing branch with more than one interior point?

A plausible route is to combine symmetry with a finite-dimensional reduction. For odd $k=2m+1$, any symmetric configuration can be written as
\[
(0^{(a_0)},y_1,\dots,y_r,\tfrac12,z_r,\dots,z_1,1^{(a_1)})
\]
with $a_0+a_1+2r+1=2m+1$. At the candidate transition one would like to show that every such branch has energy at least that of the one-midpoint branch. Since the explicit branch difference in Proposition~4.1 is linear in $m$, the real challenge is not the crossing itself but the exclusion of additional odd branches.

The numerical evidence currently available is consistent with a positive answer: in addition to $k=3,5,7,9$, the values $k=11,13,15,19$ display the same sign change at the universal odd value. This strongly suggests the following conjecture.

\begin{conjecture}
For every odd $k=2m+1$, the true global critical exponent equals the universal odd value in \eqref{eq:oddq}.
\end{conjecture}

\subsection{Is the even sequence $(q_{2m})_m$ monotone?}

For even $k=2m$ the critical exponents obtained from the symmetric two-interior-point branch are
\[
\begin{aligned}
&q_4\approx 1.06268,
\qquad
q_6\approx 1.15560,
\qquad
q_8\approx 1.20613,
\qquad
q_{10}\approx 1.23852,\\
&q_{12}\approx 1.26131,
\qquad
q_{14}\approx 1.27831,
\qquad
q_{16}\approx 1.29151,
\qquad
q_{18}\approx 1.30208,
\qquad
q_{20}\approx 1.31074.
\end{aligned}
\]
This suggests that $(q_{2m})_m$ is increasing. See also Figure \ref{fig:seq}. If true, such monotonicity would be quite natural: as $k$ grows, more colliding pairs can be created at the endpoints, while the interval budget remains fixed, so one expects endpoint-only states to become favorable for progressively smaller interior branches.

A possible proof strategy is to study the even branch energy \eqref{eq:evenbranch} after minimizing in $a$ and to compare the resulting value for successive $m$. One would like to show that the function
\[
\Phi_m(q):=\min_{0\le a\le 1/2} E_{\mathrm{even,int}}(a;q)-E_{\mathrm{even,end}}(q)
\]
changes sign exactly once and that the zero of $\Phi_m$ moves monotonically with $m$. Because the formulas are explicit, this may be accessible either analytically or via computer-assisted inequalities.

\subsection{What is the large-$k$ phase diagram in the $(k,q)$-plane?}

The family $e^{-d^q}$ appears to exhibit three regimes.
\begin{itemize}
\item For $0<q<1$, all minimizers are collision-free.
\item At $q=1$, partial endpoint clustering emerges and empirically follows a one-third law.
\item For $q>1$, endpoint-only configurations eventually dominate, with a parity effect between odd and even $k$.
\end{itemize}
A fuller phase diagram would describe, for each pair $(k,q)$, the number of endpoint-collapsed points and the dimension of the active interior block.

For large $k$, a natural goal is to identify asymptotic transition curves separating regions with different numbers of active interior points. The odd/even dichotomy already visible in Tables~\ref{tab:qk} and \ref{tab:qk-even-extended} suggests that parity survives at finite scale, but it is not yet clear whether it persists in a meaningful asymptotic form. Another attractive question is whether the threshold law for $q=1$ can be generalized to nearby values of $q$ by a perturbative analysis.

At the end of this appendix we also note that an accompanying GitHub repository with Python code for the numerical results is available at
\begin{center}
\url{https://github.com/emmerichmtm/phaseTransitionExpKernelsUnitLine}
\end{center}

\begin{figure}[t]
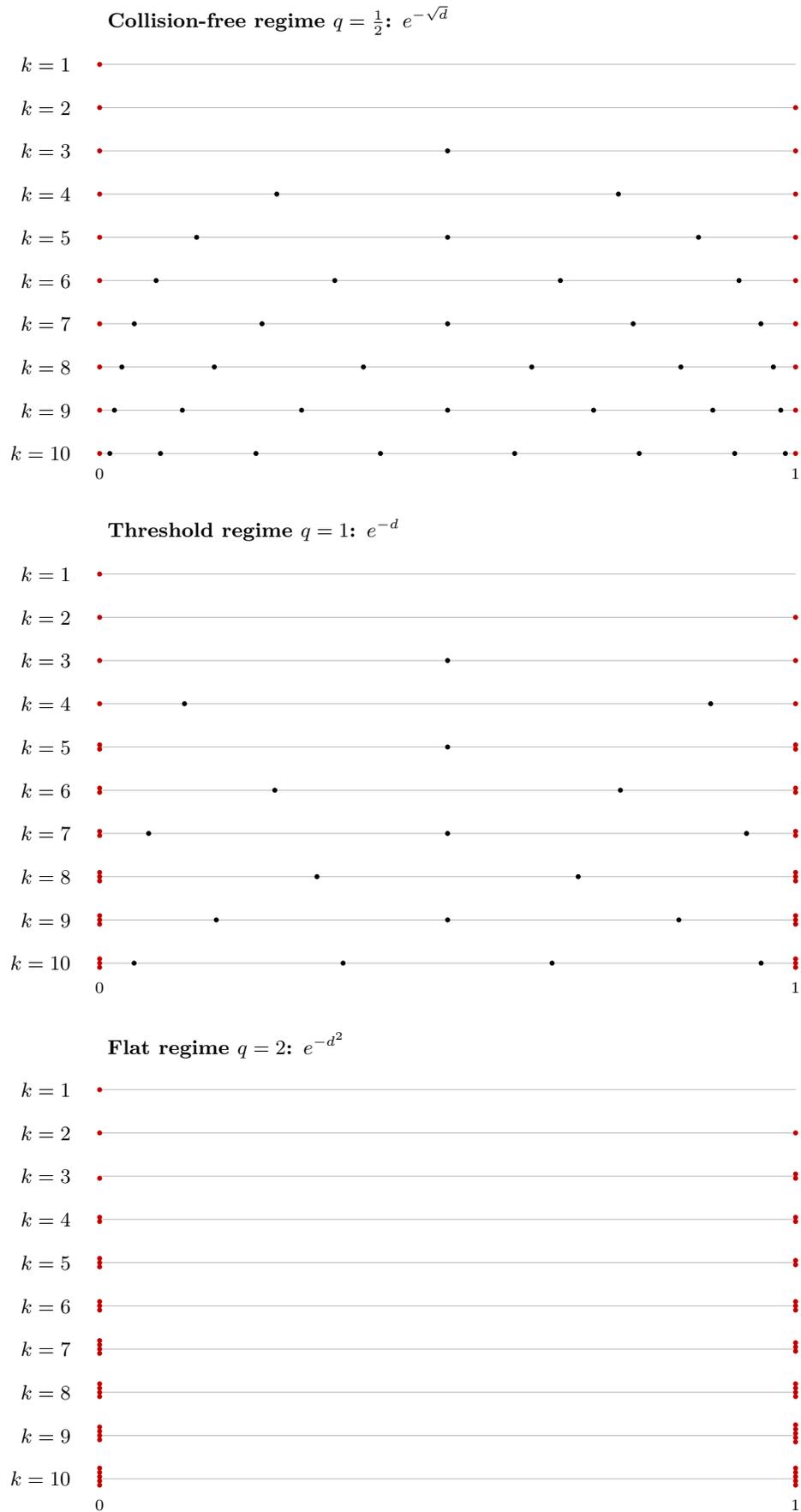

\centering

\caption{Stacked point diagrams for three representative kernels. Boundary points are drawn in dark red and stacked vertically when coincident.}
\label{fig:kernel-comparison}
\end{figure}
\begin{figure}
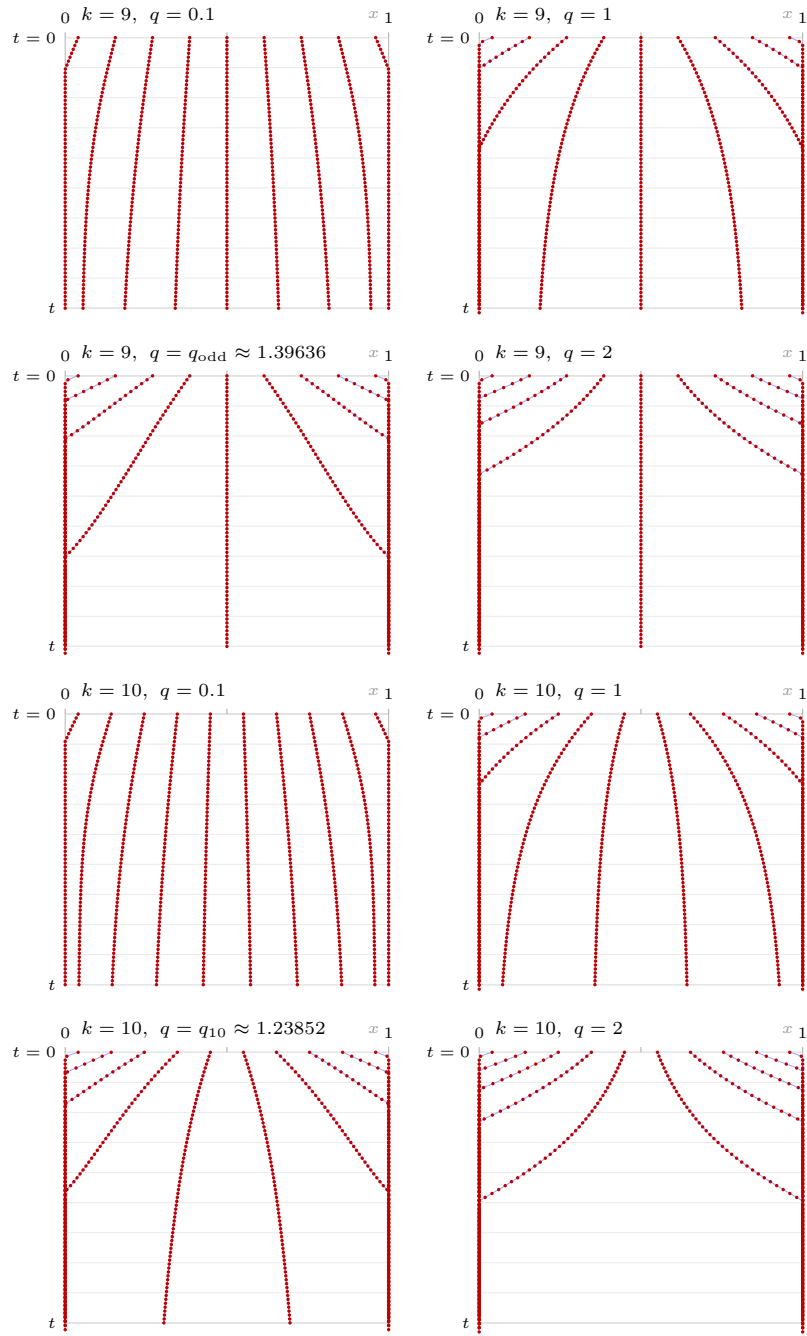

    \centering
    \caption{Gradient flow plots for different combinations of $q,k$.}
    \label{fig:gradflowqk}

%
\end{figure}
\end{document}